 \documentclass[final]{IEEEtran}

 \usepackage{hyperref}  

 \usepackage{datetime}  
\usepackage{nicematrix}
\usepackage{multicol}
\usepackage{mathtools}
\usepackage{url}
\usepackage[makeroom]{cancel}
\usepackage{soul}
 
\usepackage{amsmath,amssymb}
\usepackage{amsthm,thmtools}
\usepackage{amsfonts}
\usepackage{blkarray}

 \usepackage{graphicx}
\usepackage{enumerate}
\usepackage{color}      

\usepackage{verbatim}

\usepackage{amssymb}
\usepackage{amsbsy}
\usepackage{amsmath,lipsum}
\usepackage{mathabx}
\usepackage{setspace}
\usepackage{url}

\usepackage{float}
 \usepackage{dsfont}

\def\Aut{\mathop{\rm Aut}\nolimits}
\def\End{\mathop{\rm End}\nolimits}
\def\Hom{\mathop{\rm Hom}\nolimits}
\def\Lie{\mathop{\rm Lie}\nolimits}

 \newcommand{\Int}{\operatorname{int}}
 \newcommand{\Clos}{\operatorname{clos}}

 \newcommand{\diag}{\operatorname{diag}}

\newcommand*\diff{\mathop{}\!\mathrm{d}}

\newcommand{\such}{\, | \,}

\newtheorem{theorem}{Theorem}
\newtheorem{claim}{Claim}
\newtheorem{definition}{Definition}
\newtheorem{proposition}{Proposition}
\newtheorem{lemma}{Lemma}
\newtheorem{corollary}{Corollary}

\newtheorem{example}{Example}
\newtheorem{remark}{Remark}

\newcommand {\R}{\mathbb R}

\providecommand{\keywords}[1]
{
  \small	
  \textbf{\textit{Keywords---}} #1
}

\newcommand{\be}{\begin{equation}}
\newcommand{\ee}{\end{equation}}

\newcommand{\Tr}{\operatorname{{\mathrm Tr}}}

\usepackage{xcolor}
\usepackage{soul}

\usepackage{lineno}

\DeclareMathOperator{\cone}{cone}

\title{\LARGE \bf
On special quadratic Lyapunov functions for linear dynamical
systems with an invariant cone
  \\  \vspace*{20pt} \normalsize \date{} }
\author{Omri Dalin and Alexander Ovseevich and Michael Margaliot\thanks{The authors are with the School of Elec. Eng., Tel Aviv University, Israel 69978. Correspondence: michaelm@tauex.tau.ac.il \\
This research was partly supported by a research grant from the~Israel Science Foundation.  The work of AO was partly  supported by a research grant from the Ministry of Aliyah and Integration.} }

\begin{document}
\maketitle

\begin{abstract}
We consider a continuous-time linear time-invariant dynamical system that admits an invariant cone.  For the case of a self-dual and homogeneous cone we show that if the system is asymptotically stable then it admits a quadratic Lyapunov function with a special structure. The complexity of this Lyapuonv function scales linearly with the dimension of the dynamical system. 
In the particular case when the cone is the nonnegative orthant  this reduces to the well-known and important result that a positive    system admits a diagonal   Lyapunov function.
We demonstrate our theoretical results by deriving a new special quadratic Lyapunov function for systems that admit the ice-cream cone as an invariant set.
\end{abstract}

\keywords{Lie group, Lie algebra, stability analysis, special Lyapunov functions.}

\section{Introduction}
A set~$C $ is called a cone if~$x\in C$ implies that~$\alpha x\in C$ for all~$\alpha>0$.
Given a 
closed and convex cone~$C\subset\R^n$, 
let~$M(C)$ denote the set of matrices~$A\in\R^{n\times n}$ such that~$C$ is an invariant set
of the linear system
\be\label{eq:LTI}
\dot x(t)=Ax(t).
\ee
In other words,
\be\label{eq:xt}
x(0)\in C \implies  x(t) \in C \text{ for all } t\geq 0.
\ee
The set~$M(C)$ is sometimes referred to as the set of matrices that are exponentially nonnegative on~$C$ (see, e.g.,~\cite{stern_wolk}).

For example, consider the special case  where~$C$ is
  the cone
\[
\R^n_+:=\{
x\in\R^n \such   x_i \geq 0 ,\ i=1,\dots,n \},
\]
that is, the nonnegative orthant in~$\R^n$.
It is well-known (see, e.g.,~\cite{berman}) that~$M(\R^n_+)$ is the set of Metzler matrices, that is, the set of matrices with nonnegative off-diagonal terms.
The system~\eqref{eq:LTI}, with~$A$ Metzler, is called a positive system~\cite{farina2000} and has found many applications, as in many real-world models the state-variables represent quantities that cannot take negative values (e.g., concentrations, probabilities, densities, sizes of populations), and thus the dynamics must map~$\R^n_+$ to itself.

Recall that~$A\in\R^{n\times n}$ is called a~$Z$-matrix if~$-A$ is Metzler, that is, if~$-A\in M(\R^n_+)$. Some authors refer to the property~$-A\in M (C)$ as the~$Z$-property of~$A$ with respect to~$C$~\cite{gauda_tao2009}.

The problem of characterizing~$M(C)$
is an important  question in
linear algebra   and   dynamical systems theory.
Indeed, the existence of invariant sets has strong implications on the   behaviour of dynamical systems and also for control synthesis~\cite{BLANCHINI_set_invariance}. Moreover, the results for linear systems play an important role also in the analysis of nonlinear systems.
For example, the variational equation associated with the nonlinear system~$\dot x(t)=f(x(t))$ is
\be\label{eq:var}
\dot z(t)=J(x(t))z(t),
\ee
where~$J(x):=\frac{\partial}{\partial x}f(x)$ is the Jacobian  of the vector field~$f$.   Note that~\eqref{eq:var} is a   linear system,  and if this system admits an invariant set  then this has strong implications on the behaviour of the nonlinear system. The most important case is
when the nonnegative orthant is an invariant set of~\eqref{eq:var},  and then the nonlinear system is cooperative~\cite{hlsmith}. These systems, and their generalization  to control systems~\cite{Angeli2003monotone}, have many special properties. 

A matrix~$A\in M(C)$  has special properties. For example,
  cone preserving
operators  have special spectral properties~\cite[Chapter 11]{berman}. Here, we are interested in the case where~$A\in M(C)$ and $A$ is Hurwitz. It is well-known that then~$A$ admits special stability properties  (see, e.g.,~\cite{diffusive}). In particular, if~$A\in M(\R^n_+)$ and~$A$ is Hurwitz then it a admits  a
positive diagonal Lyapunov function, that is, there exists a positive diagonal matrix~$D$ such that~$DA+A^TD$ is negative-definite.  This property and its extensions have found many applications in systems theory, see, e.g.,~\cite{tanaka_brl,RANTZER201572}. One reason for this is that  the number of parameters in a diagonal matrix scales \emph{linearly} with the dimension of the system.  
Our goal here is to show that if~$A\in M(C)$, where~$C$ is a symmetric 
cone, but not necessarily~$\R^n_+$, and~$A$ is Hurwitz then it admits a quadratic Lyapunov function that 
scales linearly with the system dimension. 

Since the solution of~\eqref{eq:xt} is~$x(t)=\exp(At)x(0)$,
it is   natural to address  the problem of characterizing special Lyapunov functions   using  Lie-groups and Lie-algebras. Here we apply such an approach  to
derive an explicit  special quadratic Lyapunov function   for Hurwitz matrices~$A\in M(C)$, where~$C$ is a symmetric cone. The complexity of this Lyapunov  function     scales  linearly  with the dimension of the system.
In the special case where~$C=\R^n_+$, this reduces to a diagonal Lyapunov function.
We demonstrate our theoretical results by deriving  a new quadratic Lyapunov function for Hurwitz matrices~$A\in M(K_n)$, where~$K_n$
is the ice-cream cone in~$\R^n$.

The remainder of this note  is organized as follows. The next section reviews several  known  definitions  and results that are used later on. Section~\ref{sec:main} describes our   theoretical results. These are demonstrated  by deriving a new special quadratic Lyapunov function for a Hurwitz matrix that is exponentially nonnegative on the ice-cream cone.

 We use standard notation.
Vectors [matrices] are denoted by small [capital] letters. 
The canonical  basis of~$\R^n$ is~$e^1,\dots,e^n$.
The transpose of a matrix~$A$ is denoted by~$A^T$. A   matrix~$Q$ is positive-definite (denoted~$Q\succ 0$)  if it is symmetric and~$x^TQx>0$ for all~$x\not =0$, and negative-definite (denoted~$Q\prec 0$) if~$-Q\succ0$.
For a square matrix~$A$, $\Tr(A)$ is the trace of~$A$, and~$\det(A)$ is the determinant of~$A$.
For two matrices~$A\in\R^{n\times m},B\in\R^{m\times n}$, let~$\langle A,B\rangle:=\Tr (AB)$. The interior [closure]  of a set~$S$ is denoted by~$\Int(S)$ [$\Clos(S)$].
Given a set of matrices~$S$, we
write
\[
\cone(S)= \{ \alpha_1 A_1+\alpha_2 A_2+\dots+ \alpha_k A_k \such k\geq 1,A_i\in S,\,  \alpha_i\geq 0 \}
\]
for the cone generated by all finite convex combinations of the matrices in~$S$.
The Minkowski sum 
of two sets~$S_1,S_2$ is
$
S_1+S_2:=\{ a+b\such a\in S_1,b\in S_2\}.
$ 

 \section{Preliminaries}
We will demonstrate the theoretical results using the ice-cream cone  (also called   Lorentz cone,   light
 cone, and second-order cone).

\subsection{Ice-cream cone}
Recall that a cone~$C\subset\R^n$ is called \emph{ellipsoidal} if
\[
C=\{x\in\R^n \such x^TQx\leq 0, (u^n)^T x\geq 0\},
\]
where~$Q\in\R^{n\times n}$ is a symmetric matrix with eigenvalues
$
\lambda_1\geq\dots\geq\lambda_{n-1}>0>\lambda_n,
$
and corresponding eigenvectors~$u^1,\dots, u^n$ \cite{wolko1991}.

The ice-cream  cone\footnote{We use the terminology from \cite{stern_wolk}. This cone is an important  object in both physics and mathematics, and is  also known as the  light cone or the Lorentz cone.}~$K_n\subset \R^n$ is the special case
of an ellipsoidal cone obtained for 
\[
Q_n:=\diag( 1, 1,\dots, 1,-1)\in\R^{n\times n},
\]
and~$u^n=e^n$, that is,
\begin{equation}\label{lcone}
 K_n:=\{x\in\R^n
 \such \sum_{i=1}^{n-1}x_i^2\leq x_n^2,\,x_n\geq0\}.
\end{equation}

Below we demonstrate our theoretical results for 
matrices~$A\in M(K_n)$.  Since every    ellipsoidal cone is a nonsingular linear
transformation of the ice-cream cone~\cite{wolko1991}, these results can be easily generalized
to ellipsoidal cones. 

\subsection{Automorphisms and endomorphisms of a cone}

Given a  cone~$C\subset\R^n$, and a linear map $\phi:\R^n\to\R^n$ (that can be identified with a matrix in~$\R^{n\times n}$), let
\[
\phi(C):=\{\phi(x)\such x\in C\}.
\]
 A cone $C$ defines two algebraic objects.
 \begin{definition}
The group  of   automorphisms of~$C$, denoted
   $\Aut(C)$,   is the set of all linear maps (i.e., matrices) $\phi:\R^n\to\R^n$ such that $\phi(C)=C$. Let~$\Aut_0(C)$ denote the connected component of the unit element of $\Aut(C)$. 
    \end{definition}

    \begin{definition}
The semigroup  of endomorphisms of~$C$, denoted $\End(C)$, is the set of all   linear maps~$\phi:\R^n\to\R^n$ such that~$\phi(C)\subseteq C$.
 \end{definition}

More generally, given two cones $C_1$ and~$C_2$, let~$\Hom(C_1,C_2)$ denote the set of linear maps~$\phi:\R^n\to\R^n$ such that~$\phi(C_1)\subseteq C_2$.

It is well-known 
that  the group $\Aut(C)$ is a closed subgroup in the group of invertible matrices, and therefore, by   Cartan's theorem~\cite{serre} it is a  Lie group.

The semigroup~$\End(C)$ is a manifold with corners.
Let~$H $ denote either~$ \Aut(C)$ or~$  \End(C)  $.
Then a tangent cone~$T_h(H)$ is defined at any point~$h$ of any of these two objects~$H$. By definition,  $T_h(H)$ is the set of tangent vectors:~$\left.\frac{d}{dt}\right\vert_{t=0}\gamma(t)$ to smooth curves~$\gamma:[0,\varepsilon)\to H$, with~$\varepsilon>0$ and~$\gamma(0)=h$.

In the case of~$\Aut(C)$ the tangent cone $T_h(\Aut(C))$ is just the tangent space at~$h$,
and for the particular case~$h=I_n$ it is the Lie algebra:
  $\Lie \Aut(C) := T_{I_n}(\Aut(C))$.

\begin{example}\label{diagonal_case}
    For the case where~$C$ is the
     nonnegative orthant~$\R^n_+$, it is straightforward to verify that~$\Aut(\R^n_+)$ consists of matrices of the form~$\Sigma D$, where~$\Sigma $ is a  permutation matrix, and~$D \in \mathcal D_+$, with
    $$
    \mathcal D_+:=\{\diag(d_1,\dots,d_n)\such d_i>0\}.
    $$
    In particular, the connected component $\Aut_0(\R^n_+)$ of~$I_n$ in~$\Aut(\R^n_+)$ is~$\mathcal D_+ $.
    Also,   $\End(\R^n_+)$ is the semigroup (under multiplication) of matrices with nonnegative entries.
\end{example}

\begin{example}\label{exa:k_n}
For the case where~$C$ is the ice-cream cone~$K_n$, 
it follows from \cite[Theorem~3.31]{stern_wolk} that  
  if~$H=\exp(\alpha I_n+B)$, with~$\alpha \in \R$,
\[  
B=\begin{bmatrix} B_1 & b \\b^T& 0 \end{bmatrix},
\]
 $B_1$ an~$(n-1)\times(n-1)$ skew-symmetric matrix, and~$b\in \R^{n-1}$
 then~$H\in \Aut_0(K_n)$.  Moreover, any~$H\in\Aut_0(K_n)$ that is sufficiently close to $I_n$
 admits  this exponential form.   In particular, the Lie algebra $\Lie\Aut(K_n)$ consists of matrices of the form $\alpha I_n+B$.
\end{example}

\subsection{Duality theory}
The dual cone of~$C$ (see,  e.g.,~\cite{rock}) is
$$C^\vee:=\{y\in\R^n\such   y^T x \geq0 \mbox{ for all }x\in C\}.$$
In particular,
\be\label{eq:cvv}
(C^\vee)^\vee:=\{z\in\R^n\such   z^T y \geq0 \mbox{ for all }y\in C^\vee\}.
\ee
The  cone~$C$ is called  self-dual if~$C^\vee=C$.
For example,~$\R^n_+$ and~$K_n$ are self-dual.

The duality theorem~\cite{rock} asserts that if~$C$ is closed and convex then  \be\label{eq:dual}
(C^{\vee})^{\vee}=C.
\ee
 
 \begin{remark} \label{rem:using_dual}
 Eq.~\eqref{eq:dual} is useful in analyzing the sets~$\End(C)$,~$\Aut(C)$, and~$M(C)$. To demonstrate this, assume that~$x\in C$ and that~$(Ax)^T\xi <0$ for some~$\xi\in C^\vee$.
Then~\eqref{eq:cvv} and~\eqref{eq:dual} imply that~$Ax\not\in (C^\vee)^\vee=C$,
so we conclude that~$A\not \in\End(C)$.
 \end{remark} 
 
 Let $C\subseteq \R^n$ be a closed   convex cone.
Recall that a matrix~$A$ is called
\emph{cross-positive on~$C$}~\cite{Cross-Positive1970} if for any $x\in C$ and $\xi\in C^\vee$ satisfying~$  x^T \xi =0$, we have that~$  (Ax)^T \xi \geq0$. Note that if~$A$ is cross-positive on~$C$ then~$A^T$ is
cross-positive on~$C^\vee$.

 \begin{theorem}\label{thm:matzler2} (see, e.g.~\cite[Thm.~3]{Cross-Positive1970})
Let $C\subset\R^n$ be a closed   convex cone.
Then~$M(C)$ is  the set of matrices    that are cross-positive on~$C$.
\end{theorem}
\begin{remark}\label{rem:mc_cone}
Note that this implies in particular that the set of matrices~$M(C)$ is a closed convex   cone in~$\R^{n\times n}$.
\end{remark}

\begin{remark}
Let~$e^1,\dots,e^n$ denote the standard basis in~$\R^n$.
Consider~$C=\R^n_+$,
and take~$x=e^i$, $\xi=e^j$, with~$i\not = j$.
Then the condition in  Theorem~\ref{thm:matzler2} reduces  to the Metzler condition:~$a_{ ij}\geq0$ for~$i\neq j$.
\end{remark}

 The next  result provides a Lie-algebraic characterization of~$M(C)$.
For a set~$S$, let~$\Clos(S)$ denote the closure of~$S$.

\begin{theorem}\label{thm:cone_thm}
Let $C\subset\R^n$ be a closed convex cone.
Then
\be\label{eq:colsure}
M(C)=\Clos \left   ( \Lie \Aut(C)+\End(C) \right ),
\ee
where~$+$ denotes the Minkowski sum.
\end{theorem}
\begin{remark}
In general, it is not possible to remove the ``closure of'' in~\eqref{eq:colsure}. A   quite subtle counterexample, in fact a series of counterexamples, is given
in~\cite{kuzma_omladic}.  In particular, it is shown there that~$M(C)\neq \Lie \Aut(C)+\End(C)$ for cones~$C=\{x\circ x\such  x\in V\}$, where $V$ is any Euclidean Jordan algebra of rank $\geq3$.   However, it is possible to remove the 
``closure of''
in some important cases  including
the classical  case~$C=\R^n_+$, and also  for any polygonal or  ellipsoidal cone \cite{Cross-Positive1970}, \cite{stern_wolk1994}.  
\end{remark}

\begin{example}
Consider the case~$C=\R^n_+$. In this case,
$\Lie \Aut(\R^n_+)$
is the set of
all diagonal matrices,
and $\End(\R^n_+)$ is the set of matrices with nonnegative entries, so~$\Lie \Aut(\R^n_+)+\End(\R^n_+)$
is the set of Metzler matrices, that is,~$M(\R^n_+)$.
\end{example}

According to Ref.~\cite{kuzma_omladic}, the proof of Theorem~\ref{thm:cone_thm}
appears in \cite{Cross-Positive1970}. We were not able to find such a proof there, so  for the sake of completeness, we provide a simple and self-contained proof in the Appendix.

\subsection{Stability of matrices in~$M(C)$}
We now review several    known results   characterizing
matrices   in~$M(C)$ that are Hurwitz. 
We recall the following.
\begin{definition}
    A  cone~$C\subset \R^n$ is called \emph{proper} if it is closed, convex, pointed (that is,~$x\in C$ and~$-x\in C$ implies that~$x=0$),
    and has a nonempty interior.
\end{definition}
For example,~$\R^n_+$ and~$K_n$
are  proper cones.

\begin{proposition} (see, e.g.,~\cite{diffusive})
    \label{prop:stability_criterion}
 Let~$C$ be  a proper cone, and let~$A\in M(C)$. Then the following four  statements are equivalent:
 \begin{enumerate}
     \item  $A$ is Hurwitz;
     \item $-A^{-1}\in\End(C)$; 
      \item There exists $z \in \Int( C)$ such that $-Az\in \Int ( C)$;
     \item  There exists $\xi \in \Int(C^\vee)$ such that $- A^T  \xi \in \Int (   C^\vee)$.

 \end{enumerate}
\end{proposition}

For the case~$C=\R^n_+$, this  implies that if~$A$ is Metzler then the following 
statements are equivalent:
 \begin{enumerate}
     \item  $A$ is Hurwitz;
     \item All the entries of~$A^{-1}  $ are nonpositive;  
     \item There exists $z \in \Int( \R^n_+)$ such that $ -A  z\in \Int (\R^n_+)$;
     \item  There exists $\xi\in \Int(\R^n_+)$ such that $ -A^T \xi \in \Int (   \R^n_+)$,
 \end{enumerate}
(see, e.g.,~\cite{bullo_metzler}).

\section{Main Results}\label{sec:main}

This section includes our main results. These  results generalise special  stability properties of matrices that are Hurwitz and in~$M(\R^n_+)$ to matrices that are Hurwitz and in~$M(C)$.

\subsection{Invariance of stability under automorphisms of the cone}
It is well-known that
if~$A$ is Metzler and Hurwitz and~$D$ is a positive diagonal matrix then  both
$AD$ and $DA$ are also  Metzler and  Hurwitz.
This D-stability property originated in the analysis of economic models, and has found many applications (see, e.g., the survey paper~\cite{Kushel2019UnifyingMS}).
The next result generalizes this property.
\begin{proposition}
\label{prop:stab_invariance}
Let~$C$ be  a proper cone.
Suppose that $A\in M(C)$ is   Hurwitz, and
$D_1,D_2\in\Aut(C)$.
If
\be \label{eq:condd}
D_1AD_2\in M(C)
\ee
then
\[
D_1AD_2 \text{ is Hurwitz}.
\]
\end{proposition}


\begin{proof}
Since~$A$ is Hurwitz, Proposition~\ref{prop:stability_criterion}
implies that~$-A^{-1}$ maps~$C$ to~$C$.
Since $D_i^{-1}$ are automorphisms of $C$,
we obtain that
\begin{align*}
 -(D_1AD_2)^{-1}=D_2^{-1}(-A^{-1})D_1^{-1}
\end{align*}
also maps~$C$ to~$C$. Combining this with~\eqref{eq:condd} and Proposition~\ref{prop:stability_criterion} implies that~$D_1AD_2$ is Hurwitz.
\end{proof}

\begin{remark}
\label{rem:stab_invariance}
In the special case where~$D_2=D_1^{-1}$, it is immediate that 
\[
D_1AD_1^{-1} \in M(C) \text { and } D_1AD_1^{-1} \text{ is Hurwitz},
\]
as
$\exp(D_1 A D_1 ^{-1} t )=D_1 \exp(At) D_1^{-1}$. 
\end{remark}
\begin{remark}
      For the case where~$C=\R^n_+$,  $A\in M(\R^n_+)$  implies that~$A$ is Metzler, and~$D_i\in \Aut(\R^n_+)$ implies that~$D_i$ is a  positive diagonal matrix, and then~$D_1AD_2$ is also Metzler, so condition~\eqref{eq:condd} automatically  holds.  However, for other cones it is not immediately clear when~\eqref{eq:condd} indeed holds.
 \end{remark}

\begin{example}
Consider the case~$C=K_3$. 
The matrix
\[
    A=\begin{bmatrix}
        -\epsilon_1 &-1&0\\
        1&-\epsilon_1 &0\\
        0 &0 &-\epsilon_2
    \end{bmatrix},\text{ with } 0<\epsilon_2\leq\epsilon_1,
    \]
satisfies~$A \in M(K_3)$, because $Q_3A+(Q_3A)^T \preceq 2Q_3$ (see~\cite{stern_wolk}),
and~$A$ is Hurwitz  as~$A+A^T$ is a negative diagonal matrix. The matrix
\[
D(b):=\exp(\begin{bmatrix}
         0& b&0\\-b &0&0\\0 &0&0
    \end{bmatrix} ) =\begin{bmatrix}
         \cos(b)& \sin(b) &0\\-\sin(b)&\cos(b)&0\\0&0&1 
    \end{bmatrix} 
\]
satisfies~$D(b)\in\Aut(K_3)$ for any~$b\in\R$. 
Consider the matrix
\be\label{eq:dba}
D(b)A= \begin{bmatrix}
         \sin(b)-\epsilon_1 \cos(b) & -\epsilon_1\sin(b)-\cos(b) & 0\\
          \epsilon_1\sin(b) + \cos(b)& \sin(b)-\epsilon_1\cos(b) &0\\
          0&0&-\epsilon_2
    \end{bmatrix}.
\ee
This matrix belongs to~$M(K_3)$ iff 
there exists~$z\in\R$ such that~$Q_3D(b)A+(Q_3D(b)A)^T \preceq zQ_3$  (see~\cite{stern_wolk}),
that is, iff there exists~$z\in\R$ such that
\[
\diag(  z-2\sin(b)+2\epsilon_1\cos(b),  z-2\sin(b)+2\epsilon_1\cos(b), -z-2\epsilon_2    )\succeq 0.
\]
Thus,~$D(b)A \in M(K_3) $ iff
\be\label{eq:condbmk}
\sin(b)-\epsilon_1\cos(b)\leq -\epsilon_2.
\ee
 Proposition~\ref{prop:stab_invariance} implies that if~$b$   satisfies~\eqref{eq:condbmk} then~$D(b)A$ is Hurwitz.  Indeed, it is easy to verify this directly using~\eqref{eq:dba}. 
\end{example}

\subsection{Quadratic Lyapunov functions with a special structure}
It is well-known that if a matrix $A$ is  Metzler and    Hurwitz  
then it admits  a quadratic \emph{diagonal} Lyapunov function, that is, there exists a positive diagonal matrix~$D$ such that $A^TD+DA\prec  0$.

To generalize this property to matrices~$A\in M(C)$ that are Hurwitz, we need to impose more structure on the cone~$C$.

\begin{definition}
A cone~$C$ is called \emph{homogeneous}
if~$\Int(C)$ is a homogeneous space under the group $\Aut (C)$. In other words, for any two points~$x,y\in\Int(C)$   there exists~$L\in \Aut (C)$   such that~$L x=y$.
\end{definition}
\begin{definition}
    A cone that is self-dual and homogeneous is   called a symmetric cone~\cite{symmetric_cones}. 
\end{definition}

Symmetric cones are    sometimes called   domains of positivity~\cite{rothaus}.


\begin{example}
  The cone~$C=\R^n_+$ is homogeneous since
$\Int(\R^n_+)= \{x\in\R^n\such x_i>0\}$ is homogeneous under the natural action of the group of positive diagonal matrices. Indeed, for any~$x,y\in\Int(\R^n_+)$ we have
that~$\diag( y_1/x_1,\dots,y_n/x_n )   x=y$. Since~$\R^n_+$ is also self-dual, it is a symmetric cone. 
The ice-cream cone $K_n$ is also homogeneous (see~\cite[Chapter 1]{symmetric_cones}) and self-dual, so it is symmetric.
\end{example}

We can now state our main result. 

\begin{theorem}\label{vinberg_lyap0}
  Suppose  that  $C$ is a proper, homogeneous, and self-dual  cone. Let~$A\in M(C)$ be Hurwitz. Then there exists a  symmetric matrix $P$ such that $P:C\to C^\vee=C$ is an automorphism of $C$, and the quadratic form $x^TPx$ is a Lyapunov function for the dynamical system~\eqref{eq:LTI}.
\end{theorem}

\begin{proof}
The quadratic Lyapunov function 
depends on some properties of the characteristic function   associated with the cone~$C$,  denoted~$\varphi:C\to  [0,\infty]$.  Recall that 
    \be\label{eq:def_varphi}
\varphi(x):=\int_{ C^\vee} \exp(-x^T y)\diff y.
\ee
For any~$G\in \Aut(C)$,  we have 
\begin{align}\label{eq:phi_Gx}
    \varphi(Gx)& = \int_{ C^\vee} \exp(-x^T G^T y)\diff y\nonumber \\
    &=|
    \det (G^{-T}) |  \int_{ C^\vee} \exp(-x^T z )\diff z
    \nonumber\\
    &=|\det (G^{-1}) | \varphi(x), 
\end{align}
where~$G^{-T}:=(G^{-1})^T$.
In particular, for~$G=sI_n$, with~$s>0$, 
this implies that~$\varphi(sx)=s^{-n}\varphi(x)$, so~$\varphi$ is homogeneous of degree~$-n$. By Euler's theorem, 
\be\label{eq:euler}
\varphi(x)=- \frac{1}{n} (\frac{\partial \varphi(x)}{\partial x})^T x.
\ee

Define~$\phi:C\to C^\vee $ by
\be\label{eq:def_just_phi}
\phi(x):=-\frac{\partial}{\partial x} \log \varphi(x)  =  \frac{1}{\varphi(x) }\int_{ C^\vee } \exp(-x^T y) y \diff y   . 
\ee
Note that for any~$x\in C$, we have
\begin{align*}
    (\phi(x))^T x & = \frac{1}{\varphi(x) }\int_{ C^\vee } \exp(-x^T y) y^T x \diff y\\
    &\geq 0,
\end{align*}
so indeed~$\phi$ maps~$C$ to~$C^\vee$. 

Also, by defintion,~$\phi(x)=\frac{-1}{\varphi(x) }  \frac{\partial \varphi(x)}{\partial x}  $, so~\eqref{eq:euler} gives 
\[
x^T\phi(x)  = n . 
\]
For any~$G\in\Aut(C)$, we have 
\[
\phi(Gx)= -\frac{\partial}{\partial x} \log \varphi(Gx) =G^{-T} \phi(x). 
\]
In particular, for~$G=sI_n$, with~$s>0$, 
this implies that~$\phi(sx)=s^{-1}\phi(x)$, so~$\phi$ is homogeneous of degree~$-1$, and  Euler's theorem implies that 
\[
\phi(x)=P(x) x,
\]
with
\[
P(x) := - (\frac{\partial \phi(x)}{\partial x})^T.    \]
It is known 
for any~$x\in \Int(C)$, we have that~$P(x)\in \Aut(C)$, and~$P(x)$ is a positive-definite matrix  (in fact,~$P(x)$ is the Hessian of a smooth and strictly convex function on~$C$). Furthermore, for any pair of vectors~$a,b\in\Int(C)$ there exists~$z=z(a,b)\in\Int(C)$ such that~$P(z)a=b$~\cite[Lemma~3.7]{rothaus}.

Since~$A\in M(C)$ and~$A$ is Hurwitz,  Proposition~\ref{prop:stability_criterion} implies that there exist~$v,w$ such that
\begin{align} \label{eq:avw}
v,w\in\Int(C),\; Av \in -C,\; A^Tw \in -C.
\end{align}
Let~$z\in \Int(C)$ be such that~$P:=P(z)$ satisfies~$P u=w$.
It is known that~$P$ admits a positive definite square root~$Q$, that is,~$P=Q^2$,  such that~$Q\in\Aut(C)$~\cite{rothaus}. We will show that~$x^T P x$ is a Lyapunov function for~\eqref{eq:LTI}. Since~$A$ is Hurwitz, it is enough to prove 
that
\be\label{eq:LisNeg}
 PA+A^TP \prec  0. 
\ee
To show this, let
\begin{align*}  
W &:=Q^{-1}  ( PA+A^TP)   Q^{-1} \\
  &=Q AQ^{-1}+Q^{-1}A^T   Q .
\end{align*}
Then~$W$ is symmetric, and since~$Q \in\Aut (C)$ and~$A\in M(C)$, we have that~$Q AQ^{-1}\in M(C)$, $Q^{-1}A^T   Q \in M(C)$, and Remark~\ref{rem:mc_cone} implies that~$W\in M(C)$. 
Let~$r:=Q u$, and note that~$r\in \Int(C)$. Consider
\begin{align*}
    W r & = QA  u+ Q^{-1}A^T P u \\
        &= Q A u+Q^{-1} A^T w. 
\end{align*}
Combining this with~\eqref{eq:avw}, the fact that~$W \in M(C)$, and 
Proposition~\ref{prop:stability_criterion}
     implies that~$W$ is Hurwitz. Since~$W$ is  a symmetric matrix, we conclude that~$W \prec 0$.    This proves~\eqref{eq:LisNeg}.
\end{proof}

\begin{remark}
Note that in the   case $C=\R^n_+$, Theorem \ref{vinberg_lyap0} gives exactly the classical result  on existence of diagonal Lyapunov functions for Metzler matrices (\cite{Horn1991TopicsMatrixAna}, \cite{berman}), as the connected component of unit in~$\Aut(\R^n_+)$ consists of positive diagonal matrices (see Example~\ref{diagonal_case}).

More specifically, it is instructive to explicitly  calculate   some of the functions defined in the proof of Theorem~\ref{vinberg_lyap0} for the particular  case where~$C=\R^n_+$. Fix~$z\in\Int(\R^n_+)$. 
Then~$z=\diag(z_1,\dots,z_n) 1_n$, where~$1_n \in\R^n$ is a vector of all ones, and using~\eqref{eq:phi_Gx} 
gives
\[
\varphi(z) = |\det(\diag( z_1^{-1},\dots,z_n^{-1}  ))| \varphi(1_n ). 
\]
Applying~\eqref{eq:def_varphi} gives
\[
\varphi(z)=  (z_1\dots z_n)^{-1}.
\]
Now~\eqref{eq:def_just_phi} gives
$
\phi(z)=\begin{bmatrix}
    z_1^{-1}&\dots&z_n^{-1}
\end{bmatrix}^T$. 
Thus,~$\phi(z)=P(z)z$, with
\[
P(z)= \diag( z_1^{-2},\dots,z_n^{-2}  ).
\]
Note that this matrix is positive-definite, and its square root is 
$
Q(z)= \diag( z_1^{-1},\dots,z_n^{-1}  ).
$
Also, for any~$a,b\in\Int(\R^n_+)$, we have that
$P(\sqrt{a_1/b_1},\dots\sqrt{,a_n/b_n})a=b$.
\end{remark}

\begin{remark}\label{rm:dimension}
 It follows from the   proof of  
Theorem~\ref{vinberg_lyap0} that the  space of ``potential  quadratic   Lyapunov'' functions~$V(z)=z^TPz$ is homeomorphic to the space~$PC$ of  rays from the origin of the cone~$C$. Indeed, the only condition that~$P=P^T$ should satisfy is that~$Pu=w$, that is,~$\phi(u)=w$. According to \cite{vinberg}, such a condition    defines~$P$ up to multiplication by a positive constant.  
In particular, the dimension of the space of admissible $P$ is one
less than the dimension of the state space~$n$.
For example, when~$C=\R^n_+$ the Lyapunov functions constructed by means of Theorem \ref{vinberg_lyap0} are diagonal, and up to scaling  by a positive constant
they depend on~$n-1$ parameters. 
\end{remark}

There is considerable interest in finding sum-separable and max-separable 
Lyapunov functions, as this can facilitate 
building  Lyapunov functions for large-scale systems 
(see, e.g.~\cite{separable2015,max_sep} and the references therein). The next result follows immediately from the proof of  Theorem~\ref{vinberg_lyap0}.
\begin{corollary}\label{cor:split}
 Assume  that a cone $C$ is the direct sum of two   cones: $C=C_1\oplus C_2\subset \R^{n_1}\oplus \R^{n_2}$, with~$C_1,C_2$ symmetric cones. Suppose that~$A\in M(C)$ and~$A$ is Hurwitz. Then there exists a quadratic Lyapunov function for~\eqref{eq:LTI} in the form~$V(x)=x^TPx$, with
 $P=P_1\oplus P_2$, where~$P_i\in\Aut(C_i)$.
\end{corollary}

\subsection{An application: Quadratic Lyapunov functions for  systems  preserving 
the ice-cream  cone} 
The ice-cream  cone~$K_n$ in~\eqref{lcone} is  self-dual, that is,  $K_n^\vee=K_n$, so an isomorphism~$K_n \to (K_n)^\vee$
is an automorphism of~$K_n$. The ice-cream cone is also homogeneous (see~\cite[Chapter 1]{symmetric_cones}).
Consider the  system~\eqref{eq:LTI}, with~$A\in M(K_n)$ and   Hurwitz. Our goal is to apply
  Theorem \ref{vinberg_lyap0} to derive a special Lyapunov function for this system.
This requires an explicit description of the symmetric part of~$\Aut(K_n)$.
\begin{proposition}\label{Euler}
 Suppose that~$P $ belongs to
  the connected component of~$I_n$
in~$\Aut(K_n)$ and~$P=P^T$.
  Then  up to a multiplication by a positive constant,  either~$P= I_n$  or
    \be\label{eq:pnob}
P= I_n+    \begin{bmatrix}
\bar b \bar b ^T&0
\\
0&1
\end{bmatrix}\left ( \cosh (  \| b \| _2)-1 \right ) +   \begin{bmatrix}
0&\bar b
\\ \bar b^T
 &0\end{bmatrix}\sinh ( \|b\|_2),
\ee
with     $b\in\R^{n-1}\setminus \{0\}$, and
$\bar b:=\frac{b}{\|b\|_2}$.
\end{proposition}
 \begin{proof}
It follows from \cite[Theorem~3.31]{stern_wolk} that $H =H^T $
belongs to
   the connected component~$\Aut_0(K_n)$ of~$I_n$ 
iff
\[
H= c \exp\left(  \begin{bmatrix} 0 &  b \\ b^T& 0 \end{bmatrix}\right)
\]
for some $c>0$.
If~$b=0$ then this gives~$H=c I_n$. If~$b\neq0$ then this yields~\eqref{eq:pnob}.
\end{proof}

Applying  Theorem~\ref{vinberg_lyap0} yields the following result.
\begin{corollary}
Suppose that  $A\in M(K_n)$ is Hurwitz. Then, there exists a Lyapunov function $V(x):=x^TP x $ for the system~\eqref{eq:LTI}, where  either~$P=I_n$ or~$P$  is 
in  the form~\eqref{eq:pnob},
\end{corollary}

\begin{remark}
   If~$A$ is Metzler and Hurwitz then  it admits a positive diagonal Lyapunov function~$P$ and since we can scale it so that~$p_{11}=1$, it depends on~$n-1$ parameters. Interestingly, this is also the case in the Lyapunov function given in~\eqref{eq:pnob}, as it
  depends on the $(n-1)$-dimensional vector~$b$. This is a particular case of the principle  described  in Remark~\ref{rm:dimension}. 
\end{remark}

\begin{example}
    Suppose that~$n=3$.
The matrix
\[
    A=\begin{bmatrix}
        -\epsilon_1 &-1&0\\
        1&-\epsilon_1 &0\\
        0 &0 &-\epsilon_2
    \end{bmatrix}
    \]
satisfies~$A \in M(K_3)$ and~$A$ is Hurwitz iff
$
0<\epsilon_2\leq\epsilon_1.
$
Clearly,  in this case~$V(x)=x^T x$ is a Lyapunov function for~\eqref{eq:LTI}.

The matrix
\[
    A=\begin{bmatrix}
        -2.1 &  1&1 \\
        -1&-2.1 &2\\
        1 &2 &-2.1
    \end{bmatrix}
    \]
   satisfies~$A \in M(K_3)$ and~$A$ is Hurwitz. In this case,~$V=x^Tx$ is not a Lyapunov  function of~\eqref{eq:LTI}. Thus, there exists a quadratic Lyapunov function with~$P$ as in~\eqref{eq:pnob}.
   Indeed, let~$b =\begin{bmatrix}
   - 1 & 0 \end{bmatrix}^T.
   $ Then~$\|b\|_2=1 $,
    $ \bar b =b $,  and~\eqref{eq:pnob} gives
\[
P=  I_3 +
\begin{bmatrix}
 1& 0& 0 \\ 0 &0& 0 \\ 0& 0& 1
\end{bmatrix}\left ( \cosh (  1   )-1 \right ) +   \begin{bmatrix}
 0 & 0 &-1\\0&0&0\\-1&0&0
 \end{bmatrix}\sinh ( 1 ), 
\]
and it is straightforward to verify that~$PA+A^T P\prec 0$.
\end{example}

\section{Conclusion}
If~$A$ is a Metzler matrix then~$\exp(At)$, with~$t\geq 0$, maps the cone~$\R^n_+$ to itself.
If~$A$ is Metzler and Hurwitz then it has several interesting  properties. For example, it is D-stable and it
admits a diagonal quadratic  Lyapunov function. We used a Lie-algebraic approach to extend these properties to matrices~$A$ such that~$\exp(At)$, with~$t\geq 0$,  maps a cone~$C$ to itself.

These generalizations  also reveal  what are the special  properties of the cone~$\R^n_+$ (e.g., convexity, self-duality, closedness, homogeneity) that are needed to obtain the special stability properties. We demonstrated these results by deriving a special quadratic
Lyapunov function for matrices~$A$  that are Hurwitz and satisfy that~$\exp(At)$, with~$t\geq 0$,  maps the ice-cream cone to itself.

 Corollary \ref{cor:split} suggests the following research question. Suppose that a cone $C$ satisfies~$C=C_1\oplus C_2$, where~$C_1,C_2$ are cones that are not necessarily symmetric. Let   $A\in M(C)$ be Hurwitz. Does there exists a separable Lyapunov function  for~$A$?

It is known that
solutions of certain LMIs related to the Bounded Real Lemma admit a  special structure in the special  case of positive systems~\cite{tanaka_brl}. Another research direction  is  to use the approach described here to generalize these results to systems with an invariant cone.
\section*{Appendix: Proof of Theorem~\ref{thm:cone_thm}}

 The proof of Theorem~\ref{thm:cone_thm} is based on several  auxiliary  results.

\begin{claim}
Let~$A,B,C $ be  closed convex cones  in~$\R^n$ such that~$B+C$, where~$+$ denotes the  Minkowski sum,  is also  closed. Then
\be\label{eq:nedpit}
A\subseteq B+C \text{ iff }
C^\vee\cap B^\vee\subseteq  A^\vee.
\ee
\end{claim}
\begin{proof}
By definition, the dual cone of the closed and convex cone~$B+C$ is 
\[
 (B+C)^\vee=\{y\in\R^n \such  y^T (b+c)\geq0 \text{ for all }b\in B,\,c\in C\}.
\]
It is straightforward to verify that this implies that
\be\label{eq:bvcv}
 (B+C)^\vee=B^\vee\cap C^\vee. 
\ee

Now suppose that~$A\subseteq B+C$. Then~$(B+C)^\vee\subseteq A^\vee $, so~$B^\vee \cap C^\vee \subseteq  A^\vee$. This proves one implication in~\eqref{eq:nedpit}.
To prove the converse implication, assume that~$B^\vee \cap C^\vee \subseteq  A^\vee$. 
Then~$(B+ C)^\vee \subseteq  A^\vee$, so~$(A^\vee)^\vee \subseteq ((B+C)^\vee)^\vee$, and using duality gives~$A\subseteq B+C$. 
 \end{proof}

\begin{lemma}
    \label{lemma:cone_1d}
Let $C\subset\R^n$ be a closed convex cone.
Then
\[
 \Lie \Aut(C)+\End(C)\subseteq M(C).
\]
\end{lemma}
\begin{proof}
Fix~$x\in C$, $A_0\in \Lie \Aut(C)$, $A_1\in \End(C)$.
Then we need to show that
\begin{equation}\label{pf}
  x+\varepsilon (A_0+A_1) x+o(\varepsilon)x\in C
\end{equation}
 for all  sufficiently  small~$\varepsilon>0$.
Since~$x\in C$ and~$A_0\in \Lie \Aut(C)$,    we have~$(I+\varepsilon A_0+o(\varepsilon))x\in  C$.   Also,~$\varepsilon A_1 x\in C$. Summing up proves~\eqref{pf}.
\end{proof}

\begin{lemma}
    \label{lemma:cone_oher}
Let $C\subset\R^n$ be a closed convex cone.
Then
\[
M(C) \subseteq \Clos\left  (\Lie \Aut(C)+\End(C)\right )  .
\]
\end{lemma}
\begin{proof}
Define   the following sets  of matrices:
\begin{align}\label{sets}
S_M&:=\{x\xi^T\such x\in C,\,\xi\in C^\vee, \,  x^T \xi =0\},  \nonumber\\
  S_E&:=\{x\xi^T\such x\in C,\,\xi\in C^\vee\},\\
S_L&:=\{\pm x\xi^T\such x\in C,\,\xi\in C^\vee, \,  x^T \xi =0\}\nonumber . 
\end{align}
 Theorem~\ref{thm:matzler2} asserts  that~$A\in M(C)$ iff~$\langle A,B \rangle \geq0 $ for any~$B\in S_M$. 
Arguing as in Remark~\ref{rem:using_dual} implies that~$A\in\End(C)$ iff~$\langle A,B \rangle \geq0 $ for any~$B\in S_E$.
We claim   
   that~$A\in\Lie\Aut(C)$ iff
\begin{equation}\label{ALie} 
\langle A,B \rangle \geq0 \text{ for any } B\in S_L.
\end{equation}
   To show this note that 
 Theorem \ref{thm:matzler2} implies that~\eqref{ALie}
 is equivalent to~$ A\in M(C)$ and~$-A\in M(C)$, that is,
  $\exp(tA)$ maps $C$ to itself for any~$t\in\R$.   In other words, for any $t\in\R$ the matrices $\exp(tA)$ and $(\exp(tA))^ {-1}$  map~$C$ to itself. Therefore, $\exp(tA) \in\Aut(C)$, and $A\in\Lie\Aut(C)$.

The sets $M(C),\,\End(C),\,\Lie\Aut(C)$ are  cones in $\R^{n\times n}$. It follows that their dual cones are
\begin{align*}
    (M(C))^\vee &=  \cone(S_M),\\
    (\End(C))^\vee &=  \cone(S_E),\\
    (\Lie\Aut(C))^\vee &=\cone(S_L),
\end{align*}
where  $\cone(X)$ stands for the cone generated by the set~$X$. 
Note that all three dual cones are closed and convex. 
By duality, the statement $M(C)\subseteq \Clos(\End(C)+\Lie\Aut(C))$  is equivalent to
 $\cone(S_E)\cap \cone(S_L)\subseteq \cone(S_M)$, and this indeed holds by~\eqref{sets}. 
 \end{proof}

The proof  of Theorem~\ref{thm:cone_thm}
follows from Lemmas~\ref{lemma:cone_1d} and~\ref{lemma:cone_oher}.


\end{document}